\documentclass[reqno, 12pt]{amsart}
\usepackage[a4paper, margin=1in]{geometry}
\usepackage{amsthm, amsmath, amsfonts, amssymb, needspace}
\title{Revisiting Positive Exponent power weighted Stein-Weiss inequalities in \(L^1\)}
\author{Wen Qi Zhang}
\address{Mathematical Sciences Institute\\ Australian National University\\Acton ACT 2601, Australia}
\email{WenQi.Zhang@anu.edu.au}
\usepackage{mathdots}
\usepackage{microtype}
\usepackage{xcolor}
\usepackage{verbatim}
\usepackage{enumitem}
\usepackage{thmtools}
\usepackage{tikz}
\usepackage{needspace}

\usetikzlibrary{arrows.meta}

\usepackage{pgf}
\usepackage{pgfplots}
\usepgfplotslibrary{fillbetween}
\usetikzlibrary{intersections}

\newlength\tindent
\theoremstyle{plain}
\newtheorem{thm}{Theorem}[section]
\newtheorem{prop}[thm]{Proposition}
\newtheorem{lemma}[thm]{Lemma}

\newtheoremstyle{boldremark}
{\dimexpr\topsep/2\relax} 
{\dimexpr\topsep/2\relax} 
{}          
{}          
{\bfseries} 
{.}         
{.5em}      
{}          

\theoremstyle{boldremark}

\theoremstyle{definition}
\newtheorem{defn}[thm]{Definition}

\usepackage{graphicx}
\usepackage{hyperref}
\DeclareGraphicsExtensions{.eps}

\usepackage{fancyhdr}
\usepackage{mathtools}
\usepackage{mathrsfs}

\DeclarePairedDelimiter\ceil{\lceil}{\rceil}
\DeclarePairedDelimiter\floor{\lfloor}{\rfloor}

\def\N{{\mathbb N}}

\def\R{{\mathbb R}}

\def\c00{c_{00}}

\DeclarePairedDelimiter\norm{\lVert}{\rVert}

\def\I{{\mathcal I}_{\lambda}}

\begin{document}
	\maketitle
	\begin{abstract}
		We obtain the necessary and sufficient conditions on input functions for a family of positive exponent power weight \(L^1\) Stein-Weiss inequalities. Earlier work by De N\'{a}poli and Picon \cite{NapPic} showed that a certain cocanceling condition was necessary and sufficient in \(L^1\), but it was unknown if this condition remained necessary in the positive exponent power weight setting. This present article identifies weaker necessary conditions for positive exponent power weights, some of which were shown to be sufficient for positive, non-integer exponent power weights in the author's earlier work \cite{mypreprint}.
	\end{abstract}
	\section{Introduction}
	Recently, progress has been made on a family of non-negative exponent power weight Stein-Weiss inequalities in \(L^1\), with the first such example proved by De N\'{a}poli and Picon \cite{NapPic}. 
	These improvements were derived by constraining the input function, and as a result the Stein-Weiss inequality can be shown to hold for all \(f\) in \(L^1\) satisfying a so-called cocanceling condition (see Def. \ref{cocancel_defn} below). Since Stein-Weiss inequalities are known to otherwise fail in \(L^1\), these constraints can be seen as a way of repairing this failure. We contribute to this literature by confirming weaker conditions (not involving cocanceling, or other differential constraints) are both necessary and sufficient for Stein-Weiss inequalities in power weighted \(L^1\) with positive, non-integer exponents. In order to state our result, we first establish our convention for the Riesz potential, \(\I\) and we write (for \(n\geq 2\), \(0<\lambda<n\)),
	\begin{align*}
		\I f(x):=\int_{\R^n}\frac{f(y)}{|x-y|^{n-\lambda}}dy, \qquad f\in\mathcal{S}(\R^n).
	\end{align*}
	If \(\lambda\) is an even integer, \(\I\) is related to fundamental solutions of integer powers of the Laplacian, and this relationship affects our results. Firstly, we state our result in the `generic case' when \(\lambda\) is not an even integer or \(\floor{a}<\lambda\) (here \(\floor{a}\) denotes the floor of \(a\)):
		\begin{thm}\label{newSW}\(\)\\
		Let 
		\(n\geq 2\), \(0<\lambda<n\) and suppose that \(0<a\), \(q\in[1,\frac{n}{n-\lambda})\) and \(b\) are the exponents satisfying \(\frac{n+b}{q}=n-\lambda+a\). Further suppose \(a\notin \N\), then if \(\lambda\) is not an even integer or \(\floor{a}<\lambda\) the inequality
		\begin{equation}
			\left(\int_{\R^n}|x|^b|\mathcal{I}_{\lambda}f(x)|^qdx\right)^{\frac{1}{q}}\lesssim \int_{\R^n}|x|^a|f(x)|dx \label{SW}
		\end{equation}
		holds if and only if \(f\in C^{\infty}_c(\R^n)\) satisfies
		the vanishing moment assumption: For all multi-indices \(|\gamma|\leq \floor{a}\), \mbox{\(\int_{\R^n} x^{\gamma}f(x) \ dx=0\)}. 
	\end{thm}
	When \(\lambda\) is an even integer, \(\I\) becomes a fundamental solution for \(\Delta^{\frac{\lambda}{2}}\) which possesses additional cancellations. These cancellations are detected when \(\floor{a}\geq \lambda\). Hence the vanishing moment assumption can be weakened to the following:
	\needspace{9\baselineskip}
	\begin{thm}\label{newSW_even}\(\)\\
		Let 
		\(n\geq 2\), \(0<\lambda<n\) and suppose that \(0<a\), \(q\in[1,\frac{n}{n-\lambda})\) and \(b\) are the exponents satisfying \(\frac{n+b}{q}=n-\lambda+a\). Further suppose \(a\notin \N\), if \(\lambda\) is an even integer and \(\floor{a}\geq \lambda\) then the inequality
		\begin{equation*}
			\left(\int_{\R^n}|x|^b|\mathcal{I}_{\lambda}f(x)|^qdx\right)^{\frac{1}{q}}\lesssim \int_{\R^n}|x|^a|f(x)|dx 
		\end{equation*}
		holds if and only if \(f\in C^{\infty}_c(\R^n)\) satisfies the polyharmonic moment assumption: \newline 
		\mbox{\(\int_{\R^n}h(x)f(x)dx=0\)} for all polyharmonic order \(\frac{\lambda}{2}\) (i.e. \(\Delta^{\frac{\lambda}{2}}h=0\)), homogeneous degree \(\ell\leq \floor{a}\) polynomials \(h\).
	\end{thm}
	For unweighted \(L^1\) (corresponding to the \(a=0\) case in Theorem \ref{newSW} above) P. De N\'{a}poli and T. Picon \cite{NapPic} showed that the cocanceling condition, first introduced by J. Van Schaftingen \cite{cocan} was both necessary and sufficient. This condition is a generalisation of a divergence free condition and in order to state it precisely, we recall the following definition:
	\needspace{4\baselineskip}
	\begin{defn}\label{cocancel_defn}
		Given finite dimensional vector spaces \(E, \ F\), let \(L(D)\) be a homogeneous linear differential operator of order \(m\) on \(\R^n\) from \(E\) to \(F\). The operator \(L(D)\) is cocanceling if 
		\begin{align*}
			\bigcap_{\xi\in\R^n\setminus\{0\}}\ker L(\xi)=\{0_E\}.
		\end{align*}
	\end{defn}
	A simple example of the above definition is the divergence operator \(L(D)=\text{div}\). With this definition in mind,  De N\'{a}poli and Picon showed that if all \(f\) satisfying \(L(D)f=0\) also satisfied Inequality \eqref{SW} for \(a=0\), then \(L(D)\) must be cocanceling. However, it was not known if this condition remained necessary when \(0<a<1\), and our above result confirms that cocanceling can be replaced with a mean zero condition in this setting. Note that when \(0<a<1\) one has \(\floor{a}=0\), and the vanishing moment assumption in Theorem \ref{newSW} reduces to mean zero.\\ 
	\\
	We also highlight that our results extend to vector fields \(f\), taking value in some finite dimensional vector space \(E\). In this setting, we identify \(E\) with Euclidean space and so the Riesz potential acts componentwise on \(f\). We do not need any special properties of \(E\) so the scalar valued proofs translate componentwise.\\
	\\
	For readers interested about power weights with integer exponents, we remark that even stronger cancellations will be necessary. This was already observed by \cite{NapPic} who demonstrated that cocanceling was insufficient when \(a=1\). A weaker positive result was recorded in the author's earlier work \cite{mypreprint} which uses a Green's function/potential for \(D_{\text{sym}}\) in place of \(\I\), and the argument closed because of additional cancellations in this Green's function.
	\section{Necessity of Theorems \ref{newSW} and \ref{newSW_even}}
The following proposition covers necessity:
\needspace{9\baselineskip}
\begin{prop}\label{necessity_mean_0}
	Let \(n\geq 2\), \(0<\lambda<n\). Suppose that \(a>0\), \(q\in[1,\frac{n}{n-\lambda}]\) and \(b\) are parameters satisfying the scaling condition \(\frac{n+b}{q}=n-\lambda+a\). If \(f\in C^{\infty}_c(\R^n)\) satisfies inequality \eqref{SW}. Then, unless \(\lambda\) is an even integer and \(\ceil{a}-1\geq\lambda\), it must be the case that for all multi-indices \(|\gamma|\leq \ceil{a}-1\) one has
	\begin{align*}
		\int_{\R^n}x^{\gamma}f(x)dx
		&=0.
	\end{align*}
	In the exceptional case where \(\lambda\) is an even integer and \(\ceil{a}-1\geq \lambda\) one instead has that for every \(\ell\leq \ceil{a}-1\), there exists a homogeneous polynomial \(W\) of degree \(\ell-\lambda\) such that
	\begin{align}
		\sum_{|\gamma|=\ell}
		\frac{M_{\gamma}}{\gamma!}\xi^{\gamma}&=W(\xi)|\xi|^{\lambda}. \label{factored}
	\end{align}
	Here \(M_{\gamma}=\int_{\R^n}(-x)^{\gamma}f(x)dx\).
\end{prop}
In the above proposition we do not distinguish between integer and non-integer values of the parameter \(a\), which is why we use \(\ceil{a}-1\) instead of \(\floor{a}\) which is what appears in Theorems \ref{newSW} and \ref{newSW_even}. When \(a\notin\N\) these are equivalent, but our necessity proof also works for \(a\in\N\) if \(|\gamma|<a\), i.e. \(|\gamma|\leq \ceil{a}-1\). We stated Theorems \ref{newSW} and \ref{newSW_even} in terms of \(\floor{a}\) to remain consistent with our previous work \cite{mypreprint}, and because we do not yet have sufficient conditions when \(a>0\) and \(a\in\N\).\\
\\
It may not be obvious that Condition \eqref{factored} in Proposition \ref{necessity_mean_0} is the same as the polyharmonic moment condition in Theorem \ref{newSW_even}. We clarify this as a separate lemma below:\\
\begin{lemma}\label{equiv}
	Suppose that \(\lambda\) is an even integer and \(\ell\geq \lambda\). Then the following are equivalent:
	\begin{enumerate}[label=(\roman*)]
		\item\label{one} The order \(\ell\) moments of \(f\) satisfy the equation
		\begin{align*}
			\sum_{|\gamma|=\ell}\frac{M_{\gamma}}{\gamma!}\xi^{\gamma}&=W(\xi)|\xi|^{\lambda}
		\end{align*}
		for some homogeneous degree \(\ell-\lambda\) polynomial \(W\). Here \(M_{\gamma}=\int_{\R^n}(-x)^{\gamma}f(x)dx\).\\
		\item\label{two} \(\int_{\R^n}h(x)f(x)dx=0\) for all polyharmonic order \(\frac{\lambda}{2}\), homogeneous degree \(\ell\) \mbox{polynomials \(h\)}.\\ 
		\item\label{three} When viewed as linear operators on homogeneous degree \(\ell\) polynomials \begin{align*}
		\ker(\Delta^{\frac{\lambda}{2}})&\subseteq\ker\left(\sum_{|\gamma|=\ell}\frac{M_{\gamma}}{\gamma!}\partial^{\gamma}\right).
		\end{align*}
	\end{enumerate}
\end{lemma}
\begin{proof}
	To simplify notation, we let \(P(\xi)=\sum_{|\gamma|=\ell}\frac{M_{\gamma}}{\gamma!}\xi^{\gamma}\), and for a homogeneous degree \(\ell\) polynomial \(h\) we write \(h(x)=\sum_{|\alpha|=\ell}h_{\alpha}x^{\alpha}\) where \(\{h_{\alpha}\}_{|\alpha|=\ell}\) are constants. The following calculation verifies that \(P(D)h=(-1)^{\ell}\int_{\R^n}h(x)f(x)dx\) for all homogeneous degree \(\ell\) polynomials, and is useful throughout:
	\begin{align*}
		P(D)h&=\sum_{|\gamma|=\ell}\sum_{|\alpha|=\ell}\frac{M_{\gamma}}{\gamma!}h_{\alpha}\partial^{\gamma}x^{\alpha}\\
		&=\sum_{|\gamma|=\ell}M_{\gamma}h_{\gamma}.
	\end{align*}
	The last line follows by computing \(\partial^{\gamma}x^{\alpha}\). Exchanging the integral in \(M_{\gamma}\) and the sum provides
	\begin{align*}
		\sum_{|\gamma|=\ell}M_{\gamma}h_{\gamma}&=\int_{\R^n}\sum_{|\gamma|=\ell}h_{\gamma}(-x)^{\gamma}f(x)dx\\
		&=(-1)^{\ell}\int_{\R^n}h(x)f(x)dx.
	\end{align*}
	We start with \ref{one} implies \ref{two}, let \(h\) be polyharmonic of order \(\frac{\lambda}{2}\), we have from \ref{one} that 
	\begin{align*}
		P(D)h&=W(D)\Delta^{\frac{\lambda}{2}}h\\
		&=0.
	\end{align*}
	Combining this with our earlier identity \mbox{\(P(D)h=(-1)^{\ell}\int_{\R^n}h(x)f(x)dx\)} gives the desired conclusion.\\
	\\
	Next, \ref{two} implies \ref{three} also follows from the identity \mbox{\(P(D)h=(-1)^{\ell}\int_{\R^n}h(x)f(x)dx\)}. Indeed, if \(h\) is polyharmonic order \(\frac{\lambda}{2}\), homogeneous degree \(\ell\) then \ref{two} states that \(\int_{\R^n}h(x)f(x)dx=0\), and so our identity gives \(P(D)h=0\).\\
	\\
	Finally, in order to demonstrate \ref{three} implies \ref{one}, we use the Fischer inner product which is defined as 
	\begin{align*}
		\langle p,q\rangle_F&:=p(D)q(0)
	\end{align*}
	where \(p,q\) are homogeneous degree \(\ell\) polynomials. A calculation shows
	\begin{align*}
		\langle |x|^{\lambda}w(x), q\rangle_F&=\langle w(x), \Delta^{\frac{\lambda}{2}}q\rangle_F
	\end{align*}
	and so, viewing \(\Delta^{\frac{\lambda}{2}}\) as a linear operator on homogeneous degree \(\ell\) polynomials, we have that the adjoint is multiplication by \(|x|^{\lambda}\). Since \ref{three} was \(\ker(\Delta^{\frac{\lambda}{2}})\subseteq \ker(P(D))\), one has for all \(h\in\ker(\Delta^{\frac{\lambda}{2}})\), \(P(D)h=0\). This implies
	\begin{align*}
		\langle P, h \rangle_F&=0
	\end{align*} 
	and hence \(P\) is orthogonal to \(\ker(\Delta^{\frac{\lambda}{2}})\). Noting that the orthogonal complement of \(\ker(\Delta^{\frac{\lambda}{2}})\) is the image of multiplication by \(|x|^{\lambda}\) we conclude \(P(x)=|x|^{\lambda}W(x)\) for some homogeneous degree \(\ell-\lambda\) polynomial \(W\), as desired.\\
\end{proof}
Before we discuss the proof of Proposition \ref{necessity_mean_0} (this will be later in Section \ref{proof_section_necessity}), we first collect three helpful lemmas. 
\section{Three Useful Lemmas}
We find the following auxiliary lemmas convenient. The first lemma quantifies the behaviour of the kernel \(|\cdot|^{\lambda-n}\) in terms of its Taylor expansion around \(0\). This idea already appears in the work of Bousquet and Van Schaftingen \cite{can}, and De N\'{a}poli and Picon \cite{NapPic} but we include it here for convenience: 
\needspace{4\baselineskip}
\begin{lemma}\label{Error_term_lemma}
	Suppose that \(\lambda<n\), \(a>0\), \(q\in[1,\frac{n}{n-\lambda}]\) and \(b\) are parameters satisfying the scaling condition \(\frac{n+b}{q}=n-\lambda+a\). Let \(R>0\), and suppose that \(y\in B(0,\frac{R}{2})\) and \(x\in B(0,2R)\setminus B(0,R)\). Then for all \(a>0\), set \(m:=\ceil{a}-1\) one has the estimate
	\begin{align}
		\left||x-y|^{\lambda-n}-\sum_{|\alpha|\leq m}(-y)^{\alpha}\frac{\partial^{\alpha}|x|^{\lambda-n}}{\alpha!}\right|&\lesssim |y|^aR^{\lambda-n-a} 
		\label{error_estimate}
	\end{align}
\end{lemma}
\begin{proof}[Proof of Lemma \ref{Error_term_lemma}]\(\)\\
	This follows by Taylor expanding \(|x-y|^{\lambda-n}\), which immediately gives the naive estimate
	\begin{align*}
		\left||x-y|^{\lambda-n}-\sum_{|\alpha|\leq m}(-y)^{\alpha}\frac{\partial^{\alpha}|x|^{\lambda-n}}{\alpha!}\right|&\lesssim |y|^{m+1}|z|^{\lambda-n-m-1}
	\end{align*}
	for some \(z\) on the line segment between \(x\) and \(x-y\). We estimate the \(|z|\) term (which has a negative exponent) first: By using the fact that \(|x|>2|y|\) observe that any point on the line segment between \(x\) and \(x-y\) must satisfy both \(|y|<|z|\) and \(\frac{|x|}{2}\leq |z|\). Furthermore, since \(|x|\geq R\) it follows that \(|z|\gtrsim R\) and so 
	\eqref{error_estimate} would be implied by
	\begin{align*}
		|y|^{m+1}R^{\lambda-n-m-1}\lesssim |y|^aR^{\lambda-n-a}.
	\end{align*}
	This follows by using that \(|y|\lesssim R\), and \(a\leq m+1\) to deduce \(|y|^{m+1-a}\leq R^{m+1-a}\). Hence
	the desired estimate follows.\\
\end{proof}
The previous lemma provides an asymptotic formula for the kernel of \(\I\), and we use this to establish that if \(\I f\) satisfies a weighted \(L^q\) norm bound, then so does its approximation on large annuli:
\begin{lemma}\label{asymptotic_L^q}
	Let \(n\geq 2\), \(0<\lambda<n\). Suppose that \(a>0\), \(q\in[1,\frac{n}{n-\lambda}]\) and \(b\) are parameters satisfying the scaling condition \(\frac{n+b}{q}=n-\lambda+a\). Let \(f\in C^{\infty}_c(\R^n)\) satisfy the inequality
	\begin{align*}
		\norm{\I f}_{L^q(|x|^b, dx)} &\lesssim \int_{\R^n}|y|^a |f(y)|dy
	\end{align*}
	Set \(A_R:=B(0,2R)\setminus B(0,R)\), \(R>2\sup\{|y|:y\in \emph{supp}f\}\), \(M_{\alpha}:=\int_{\R^n}(-y)^{\alpha}f(y)dy\), and \(M_0:=\int_{\R^n}f(y)dy.\) Then one has
		\begin{align}
		\left|\left|\sum_{|\alpha|\leq m}\frac{\partial^{\alpha}}{\alpha!}|x|^{\lambda-n}M_{\alpha}\right|\right|_{L^q(|x|^b,A_R)} \ &\lesssim \int_{\R^n}|y|^a |f(y)|dy \label{asymptotic_L^q_bound}
	\end{align}
\end{lemma}
\begin{proof}
	Again let \(m:=\ceil{a}-1\), note the immediate identity
	\begin{align}
		\sum_{|\alpha|\leq m}(-y)^{\alpha}\frac{\partial^{\alpha}}{\alpha!}|x|^{\lambda-n}f(y)&=\sum_{|\alpha|\leq m}(-y)^{\alpha}\frac{\partial^{\alpha}}{\alpha!}|x|^{\lambda-n}f(y)-|x-y|^{\lambda-n}f(y)+|x-y|^{\lambda-n}f(y) \label{splitting}
	\end{align}
	and if \(m=0\) we use the convention
	\begin{align*}
		\sum_{|\alpha|\leq 0}(-y)^{\alpha}\frac{\partial^{\alpha}}{\alpha!}|x|^{\lambda-n}&=|x|^{\lambda-n}.
	\end{align*}
	We use the following notation for simplicity
	\begin{align*}
		K(x,y)&
		:=|x-y|^{\lambda-n}-\sum_{|\alpha|\leq m}(-y)^{\alpha}\frac{\partial^{\alpha}}{\alpha!}|x|^{\lambda-n}
	\end{align*}
	and we also define
	\begin{align*}
		\mathcal{K}f(x)&:=\int_{\R^n}K(x,y)f(y)dy\\
	\end{align*}
	If we integrate both sides of \eqref{splitting} in \(y\), and then take weighted \(L^q\) norm of the resulting function of \(x\) on the annulus \(A_R:=B(0,2R)\setminus B(0,R)\), \(R>0\) we arrive at
	\begin{align}
		\left|\left|\sum_{|\alpha|\leq m}\frac{\partial^{\alpha}}{\alpha!}|x|^{\lambda-n}M_{\alpha}\right|\right|_{L^q(|x|^b,A_R)} \ &\leq \norm{\I f}_{L^q(|x|^b,A_R)}+\norm{\mathcal{K}f}_{L^q(|x|^b,A_R)}. \label{useful_bound}
	\end{align}
	The right hand side is controlled by \(\int_{\R^n}|y|^a|f(y)|dy\), first inequality \eqref{SW} would imply
	\begin{align*}
		\norm{\I f}_{L^q(|x|^b,A_R)}\lesssim \int_{\R^n}|y|^a|f(y)|dy
	\end{align*}
	and so we just need to verify
	\begin{align}
		\norm{\mathcal{K}f}_{L^q(|x|^b,A_R)}&\lesssim \int_{\R^n}|y|^a|f(y)|dy \label{key_error_term}.
	\end{align}
	This follows from Lemma \ref{Error_term_lemma} via the following argument: 
	Recall that \(x\in A_R\). Restricting \(R\) to \(\frac{R}{2}>\sup\{|y|:y\in \text{supp} f\}\), which is finite since \(f\) is compactly supported, we observe that every \(y\in \text{supp} f\) and \(x\in A_R\) meet the requirement of Lemma \ref{Error_term_lemma} and so we conclude
	\begin{align*}
		|K(x,y)|\lesssim |y|^aR^{\lambda-n-a}.
	\end{align*}
	Estimating \(\norm{\mathcal{K}f}_{L^q(|x|^b,A_R)}\) using Minkowski's integral inequality and the above bound, we conclude
	\begin{align*}
		\norm{\mathcal{K}f}_{L^q(|x|^b,A_R)}&\lesssim R^{\lambda-n-a}\int_{\R^n}|y|^a|f(y)|dy\left(\int_{A_R}|x|^bdx\right)^{\frac{1}{q}}\\
		&\lesssim R^{\lambda-n-a}R^{\frac{b+n}{q}}\int_{\R^n}|y|^a|f(y)|dy.
	\end{align*}
	Using the scaling condition, one has \(\frac{n+b}{q}=n-\lambda+a\) and so the \(R\) terms cancel which verifies \eqref{key_error_term} for all \(R\) sufficiently large.\\
\end{proof}
Our last useful lemma concerns the linear independence of the partial derivatives. Our previous lemmas combine to provide an estimate on the size of partial derivatives of \(|x|^{\lambda-n}\), and this size estimate is used in our later proof to show that if a certain homogeneous polynomial were nonzero, then it would grow too fast. In order to conclude that our desired homogeneous polynomial can only have zero as its coefficients we need to first verify that these functions are linearly independent, which is the content of our next lemma. This lemma is well known to experts, but we record it here for completeness:
\needspace{7\baselineskip}
\begin{lemma}\label{linear_independence}
	Suppose that \(0<\lambda<n\) and assume
	\begin{align}
		\sum_{|\alpha|=\ell}c_{\alpha}\partial^{\alpha}|x|^{\lambda-n}&=0\label{look_for_trivial}
	\end{align} 
	for all \(x\in\R^n\setminus \{0\}\). Then, unless \(\lambda\) is an even integer and \(\ell\geq\lambda\), one has 	
	\begin{align}
		c_{\alpha}&=0 \label{conc_2}
	\end{align}
	for all \(|\alpha|=\ell\). In the exceptional case where \(\lambda\) is an even integer, and \(\ell\geq \lambda\), there exists a homogeneous degree \(\ell-\lambda\) polynomial \(W\) such that
	\begin{align}
		\sum_{|\alpha|=\ell}c_{\alpha}\xi^{\alpha}&=W(\xi)|\xi|^{\lambda} \label{conc_1b}.
	\end{align}
\end{lemma}
\begin{proof}\(\)\\
	The assumption \eqref{look_for_trivial} on \(\R^n\setminus\{0\}\) implies that the distribution on the left hand side is supported only at \(\{0\}\). Hence, (see Theorem 2.3.4 in \cite{hormander}) it is a finite linear combination (possibly zero) of derivatives of the delta distribution \(\delta_0\). Explicitly, for some \(k\in \N\) and constants \(\{w_{\beta}\}_{|\beta|\leq k}\) \eqref{look_for_trivial} is equivalent to
	\begin{align}
		\sum_{|\alpha|=\ell}c_{\alpha}\partial^{\alpha}|x|^{\lambda-n}&=\sum_{|\beta|\leq k}w_{\beta}\partial^{\beta}\delta_0(x) \label{look_for_trivial_2}
	\end{align}
	as distributions on \(\R^n\). It is easier to analyse the distributional Fourier transform of the above identity. Since \(0<\lambda<n\),  the distribution \(|x|^{\lambda-n}\) and all of its derivatives have a well defined (distributional) Fourier transform on \(\R^n\), which can be identified with the following function (see Lemma 2 in Section 1, Chapter 5 in \cite{stein70}):
	\begin{align*}
		\widehat{|\cdot|^{\lambda-n}}&=C_{n,\lambda}|\xi|^{-\lambda}
	\end{align*}
	for some dimensional constant \(C_{n,\lambda}\), which we may safely ignore. With this in mind, \eqref{look_for_trivial_2} becomes 
	\begin{align*}
		\sum_{|\alpha|=\ell}c_{\alpha}\xi^{\alpha}|\xi|^{-\lambda}&=\sum_{|\beta|\leq k}w_{\beta}\xi^{\beta}.
	\end{align*}
	After multiplying both sides by \(|\xi|^{\lambda}\), we arrive at the following equation
		\begin{align*}
		\sum_{|\alpha|=\ell}c_{\alpha}\xi^{\alpha}&=\sum_{|\beta|\leq k}w_{\beta}\xi^{\beta}|\xi|^{\lambda}.
	\end{align*}
	The left hand side is either zero, or homogeneous of degree \(\ell\), the right hand side must also be zero, or homogeneous of degree \(\ell\). This gives \(\ell\geq \lambda\) and
	\begin{align}
		\sum_{|\alpha|=\ell}c_{\alpha}\xi^{\alpha}&=\sum_{|\beta|= k}w_{\beta}\xi^{\beta}|\xi|^{\lambda} \label{W_defn}
	\end{align}
	where \(\ell=k+\lambda\). Furthermore, testing \(-\xi\) in the above equation reveals that \(\ell\) and \(k\) must have the same parity, or both sides are \(0\). Hence both sides are non-zero precisely when \(\ell\geq \lambda\) and \(\lambda\) is an even integer. In this exceptional case, \eqref{W_defn} is the desired conclusion.\\
	\\
	Suppose now that \(\ell<\lambda\) or \(\lambda\) is not an even integer. Then both sides of \eqref{W_defn} must be zero. By the linear independence of polynomials we are forced to conclude \(c_{\alpha}=0\) for all \(|\alpha|=\ell\), and so \eqref{conc_2} follows.
\end{proof}
\section{Proof of Proposition \ref{necessity_mean_0}}\label{proof_section_necessity}
Using the lemmas discussed in the previous section, we present the proof of Proposition \ref{necessity_mean_0}:
\begin{proof}[Proof of Proposition \ref{necessity_mean_0}]\(\)\\
	Assume for a contradiction that there exists a \(\gamma\) of minimal degree such that
	\begin{align*}
		M_{\gamma}:=\int_{\R^n}(-y)^{\gamma}f(y)dy
	\end{align*}
	is nonzero. That is, if there exists another \(\beta\) such that \(M_{\beta}\neq 0\) then
	\begin{align*}
		|\gamma|\leq |\beta|.
	\end{align*}
	Let \(\ell:=|\gamma|\), \(m=\ceil{a}-1\) and assume that \(\ell\leq m\) (otherwise there is nothing to prove). By assumption, since every multi-index of degree less than \(\ell\) is zero, we begin by observing that
	\begin{align*}
		\sum_{|\alpha|=\ell}\frac{\partial^{\alpha}}{\alpha!}|x|^{\lambda-n}M_{\alpha}&=\sum_{|\alpha|\leq m}\frac{\partial^{\alpha}}{\alpha!}|x|^{\lambda-n}M_{\alpha}-\sum_{\ell<|\alpha|\leq m}\frac{\partial^{\alpha}}{\alpha!}|x|^{\lambda-n}M_{\alpha}.
	\end{align*}
	Taking weighted \(L^q\) norms on \(A_R\) and applying Lemma \ref{asymptotic_L^q} to the first term on the right hand side, we conclude
	\begin{align}
			\left|\left|\sum_{|\alpha|=\ell}\frac{\partial^{\alpha}}{\alpha!}|x|^{\lambda-n}M_{\alpha}\right|\right|_{L^q(|x|^b, A_R)}&\lesssim \int_{\R^n}|y|^a|f(y)|dy+\sum_{\ell<|\alpha|\leq m}R^{\frac{n+b}{q}+\lambda-n-|\alpha|}|M_{\alpha}|. \label{intermediate_bound}
	\end{align}
	In order to leverage Lemma \ref{linear_independence}, set
	\begin{align*}
		Q(x):=\sum_{|\alpha|=\ell}\frac{\partial^{\alpha}}{\alpha!}|x|^{\lambda-n}M_{\alpha}.
	\end{align*}
	Since \(\partial^{\alpha}|x|^{\lambda-n}\) is a homogeneous degree \(\lambda-n-\ell\) function, so is \(Q\). With this in mind, 
	we compute the left hand side of \eqref{intermediate_bound} using spherical coordinates:
	\begin{align*}
		\left|\left|\sum_{|\alpha|=\ell}\frac{\partial^{\alpha}}{\alpha!}|x|^{\lambda-n}M_{\alpha}\right|\right|_{L^q(|x|^b, A_R)}^q
		&=\int_{R}^{2R}\int_{S^{n-1}}r^{b+n-1+q(\lambda-n-\ell)}|Q(\omega)|^qd\omega dr\\
		&\gtrsim R^{b+n+q(\lambda-n-\ell)}\int_{S^{n-1}}|Q(\omega)|^qd\omega.
	\end{align*}
	Inserting this lower bound into \eqref{intermediate_bound}, and using the scaling condition to simplify the powers of \(R\), we conclude
	\begin{align*}
		R^{a-\ell}\left(\int_{S^{n-1}}|Q(\omega)|^qd\omega\right)^{\frac{1}{q}}&\lesssim \int_{\R^n}|y|^a|f(y)|dy+\sum_{\ell<|\alpha|\leq m}R^{a-|\alpha|}|M_{\alpha}|.
	\end{align*}
	Note that the largest power of \(R\) appears on the left hand side, and hence for the inequality to be valid for all large \(R\) one must have
	\begin{align}
		\int_{S^{n-1}}|Q(\omega)|^qd\omega&=0. \label{vanished_poly}
	\end{align}
	This establishes that \(Q(x)=0\) on \(\R^n\setminus\{0\}\) and so Lemma \eqref{linear_independence} applies. We conclude that, unless \(\lambda\) is an even integer and \(\ell\geq \lambda\), we must have \(M_{\alpha}=0\) for all \(|\alpha|=\ell\). When \(\lambda\) is even but \(\ell<\lambda\) we still have \(M_{\alpha}=0\) for all \(|\alpha|=\ell\). Lastly, when both \(\lambda\) and \(\ell\geq \lambda\), 
	we conclude that there exists a homogeneous polynomial \(W\) of degree \(\ell-\lambda\) such that
	\begin{align*}
		\sum_{|\alpha|=\ell}\frac{M_{\alpha}}{\alpha!}\xi^{\alpha}&=W(\xi)|\xi|^{\lambda}.
	\end{align*} 
	Furthermore, the same conclusion can be obtained for order \(k\) moments for any \(\ell<k\leq m\) by modifying the previous argument in the following manner: Instead of requiring a non-zero minimal degree moment, we suppose for a contradiction that there is a minimal \(k\) satisfying \(\ell<k\leq m\) such that the rational function
	\begin{align*}
		\sum_{|\alpha|=k}\frac{M_{\alpha}}{\alpha!}\xi^{\alpha}|\xi|^{-\lambda}
	\end{align*}
	is not a homogeneous polynomial. We use the identity
	\begin{align}
		\sum_{|\alpha|=k}\frac{\partial^{\alpha}}{\alpha!}|x|^{\lambda-n}M_{\alpha}&=\sum_{|\alpha|\leq m}\frac{\partial^{\alpha}}{\alpha!}|x|^{\lambda-n}M_{\alpha}-\sum_{k<|\alpha|\leq m}\frac{\partial^{\alpha}}{\alpha!}|x|^{\lambda-n}M_{\alpha}-\sum_{\ell\leq|\alpha|< k}\frac{\partial^{\alpha}}{\alpha!}|x|^{\lambda-n}M_{\alpha} \label{identity_2}
	\end{align}
	and note that by the minimality on \(k\) assumption, there exists a polynomial \(R\) of degree strictly lower than \(k-\lambda\) (obtained by summing the homogeneous polynomials) such that
	\begin{align*}
		\sum_{\ell\leq|\alpha|< k}\frac{\partial^{\alpha}}{\alpha!}|x|^{\lambda-n}M_{\alpha}&=R(D)\Delta^{\frac{\lambda}{2}}|x|^{\lambda-n}.
	\end{align*}
	The distribution on the right hand side is \(R(D)\delta_0\), and this vanishes when tested against any function supported away from the origin. Hence we conclude
	\begin{align*}
		\left|\left|\sum_{\ell\leq|\alpha|< k}\frac{\partial^{\alpha}}{\alpha!}|x|^{\lambda-n}M_{\alpha}\right|\right|_{L^q(|x|^b, A_R)}&=0.
	\end{align*}
	For the remaining terms in \eqref{identity_2}, we repeat the arguments leading to \eqref{intermediate_bound} and we obtain
	\begin{align*}
			\left|\left|\sum_{|\alpha|=\ell}\frac{\partial^{\alpha}}{\alpha!}|x|^{\lambda-n}M_{\alpha}\right|\right|_{L^q(|x|^b, A_R)}&\lesssim \int_{\R^n}|y|^a|f(y)|dy+\sum_{k<|\alpha|\leq m}R^{\frac{n+b}{q}+\lambda-n-|\alpha|}|M_{\alpha}|.
	\end{align*}
	Following the same steps as before to invoke Lemma \ref{linear_independence}, we arrive at a contradiction. Hence for all \(\lambda\leq \ell \leq m\), the function \(\sum_{|\alpha|=\ell}\frac{M_{\alpha}}{\alpha!}\xi^{\alpha}|\xi|^{-\lambda}\) must be a homogeneous polynomial.
\end{proof}
\section{Sufficiency of the Polyharmonic Moment Assumption}\label{proof_suffice}
To complete the proof of Theorem \ref{newSW_even} we need to verify sufficiency (when \(a\notin\N\)) in the exceptional case where \(0<\lambda<n\) is an even integer and \(\floor{a}\geq\lambda\). The other cases are covered by the author's previous work (Theorem 1.1) \cite{mypreprint}. With this in mind, it remains to verify:
\begin{prop}\label{suff_prop}\(\)\\
	Let 
	\(n\geq 2\), \(0<\lambda<n\) and suppose that \(0<a\), \(q\in[1,\frac{n}{n-\lambda})\) and \(b\) are the exponents satisfying \(\frac{n+b}{q}=n-\lambda+a\). Further suppose \(a\notin\N\), if \(\lambda\) is an even integer and \(\floor{a}\geq \lambda\) then the inequality
	\begin{equation*}
		\left(\int_{\R^n}|x|^b|\mathcal{I}_{\lambda}f(x)|^qdx\right)^{\frac{1}{q}}\lesssim \int_{\R^n}|x|^a|f(x)|dx
	\end{equation*}
	holds, provided that \(f\in C^{\infty}_c(\R^n)\) satisfies the polyharmonic moment assumption:\newline \mbox{\(\int_{\R^n}h(x)f(x)dx=0\)} for all polyharmonic order \(\frac{\lambda}{2}\), homogeneous degree \(\ell\leq\floor{a}\) \mbox{polynomials \(h\)}.
\end{prop}
\subsection{Technical Reductions}
	It is convenient to recall a technical reduction (Lemma 3.1 from \cite{mypreprint}). We recall a simplified version suitable for our purposes (but omit the proof):
	\needspace{8\baselineskip}
\begin{lemma}\label{technical_reduction}
	Let 
	\(n\geq 2\), \(0<\lambda<n\) and suppose that \(a\geq 0\), \(b\in \R\), and \(q\in[1,\frac{n}{n-\lambda})\) satisfy \(\frac{n+b}{q}=n-\lambda+a\). Let \(m=\floor{a}\) denote the floor of \(a\). If \(\rho\in L^{\infty}(\R)\) is identically \(1\) on \(\left[0,\frac{1}{2r}\right]\) and compactly supported in \(\left[0,\frac{1}{r}\right]\) for some \(r>1\), and if \(G\in C^{m+1}(\R^n\setminus\{0\})\) is positively homogeneous of degree \(\lambda-n\), define 
	\begin{align*}
		H(x,y)&:=\rho\left(\frac{|y|}{|x|}\right)\sum_{|\alpha|\leq m}(-y)^{\alpha}\frac{\partial_x^{\alpha}G(x)}{\alpha!}.
	\end{align*}
	Then the inequality
	\begin{align*}
		\left(\int_{\R^n}|x|^b \left|G*f(x)\right|^qdx\right)^{\frac{1}{q}}&\lesssim \int_{\R^n}|x|^a|f(x)|dx+\left(\int_{\R^n}|x|^{b}\left|\int_{\R^n}H(x,y)f(y)dy\right|^q dx\right)^{\frac{1}{q}}
	\end{align*}
	holds for all \(f\in C^{\infty}_c(\R^n)\) with implied constant independent of \(f\) (here \(*\) denotes convolution).
\end{lemma}
With this lemma in mind, the remaining sufficiency part of Theorem \ref{newSW_even} can be proven in a similar spirit as \cite{mypreprint}, which we provide below.
\subsection{Proof of Proposition \ref{suff_prop}}
\begin{proof}[Proof of Proposition \ref{suff_prop}]
To simplify notation we set \(\rho=\chi_{[0,\frac{1}{4}]}\). Let \(m:=\floor{a}\) and write
\begin{align*}
	H(x,y)&=\rho\left(\frac{|y|}{|x|}\right)\sum_{|\alpha|\leq m}(-y)^{\alpha}\frac{\partial_x^{\alpha}(|x|^{\lambda-n})}{\alpha!}.
\end{align*}
As before, when \(m=0\) we mean \(\sum_{|\alpha|\leq m}(-y)^{\alpha}\frac{\partial_x^{\alpha}(|x|^{\lambda-n})}{\alpha!}=|x|^{\lambda-n}\). By Lemma \ref{technical_reduction} (with \(G(x)=|x|^{\lambda-n}\)) it suffices to check
\begin{align}
	\left(\int_{\R^n}|x|^{b}\left|\int_{\R^n}H(x,y)f(y)dy\right|^q dx\right)^{\frac{1}{q}}&\lesssim \int_{\R^n}|y|^a|f(y)|dy \label{Hbound}.
\end{align}
To simplify notation we write
\begin{align*}
	J(|x|)&:= \int_{\R^n}H(x,y)f(y)dy\\
	\tilde{J}(|x|)&:= \int_{\R^n}\sum_{|\alpha|\leq m}(-y)^{\alpha}\frac{\partial_x^{\alpha}(|x|^{\lambda-n})}{\alpha!}f(y)dy
\end{align*}
and we divide our analysis into two terms, \(|J-\tilde{J}|\), and \(|\tilde{J}|\). For the difference term, we observe
\begin{align*}
	|J(|x|)-\tilde{J}(|x|)|&=\left|\int_{\R^n}\left(1-\rho\left(\frac{|y|}{|x|}\right)\right)\sum_{|\alpha|\leq m}(-y)^{\alpha}\frac{\partial_x^{\alpha}(|x|^{\lambda-n})}{\alpha!}f(y)dy\right|.
\end{align*}
The support of \(1-\rho\left(\frac{|y|}{|x|}\right)\) is the region \(4|y|>|x|\), and combining the above identity for \(|J-\tilde{J}|\) with Minkowski's integral inequality (and triangle inequality) gives,
\begin{align*}
	\left(\int_{\R^n}|x|^{b}|J(|x|)-\tilde{J}(|x|)|^{q}dx\right)^{\frac{1}{q}}&\lesssim\sum_{\ell\leq m} \int_{\R^n}|f(y)||y|^{\ell}\left(\int_{\{y\in\R^n:4|y|>|x|\}}|x|^{b+q(\lambda-n-\ell)}dx\right)^{\frac{1}{q}}dy.
\end{align*}
The integral in \(x\) on the RHS is bounded by a constant multiple of \(|y|^{a-\ell}\), to see this use the scaling condition to replace \(b\) in the exponent of \(|x|\) and realise that \(|x|^{q(a-\ell)-n}\) is locally integrable because \(a>\ell\). This means \eqref{Hbound} reduces to analysing the subtracted quantity \(\tilde{J}\). First, we realise that the polyharmonic moment assumption provides (via Lemma \ref{equiv})
a polynomial \(R\) (obtained by adding the homogeneous polynomials) which satisfies
\begin{align*}
	\int_{\R^n}\sum_{|\alpha|\leq m}(-y)^{\alpha}\frac{\partial_x^{\alpha}(|x|^{\lambda-n})}{\alpha!}f(y)dy&=R(D)\Delta^{\frac{\lambda}{2}}|x|^{\lambda-n}.
\end{align*}
When viewed as distributions, the term on the right is \(R(D)\delta_0(x)\). Using the scaling condition, for any \(0\leq \ell\leq m\) we deduce that \(\frac{b}{q}-\ell+\lambda=a-\ell+n-\frac{n}{q}\) is non-negative, since \(a>\ell\) and \(1\leq q\). This implies that
\begin{align*}
	|x|^{\frac{b}{q}}R(D)\delta_0(x)&=0
\end{align*}
in the distributional sense. Hence, interpreting the \(L^q\) norm using duality, we conclude
\begin{align*}
	\norm{|\cdot|^{\frac{b}{q}} \ R(D)\delta_0}_{L^q}&=
	0.
\end{align*}
\end{proof}
\section{AI Usage}
The author used AI tools to check an earlier version of this work. This check highlighted a subtlety missed by the author concerning the case when \(\lambda\) is an even integer in Lemma \ref{linear_independence}. The author benefited from insights obtained from digesting this discussion, and has modified their proof accordingly. No unverified output has been retained and the author accepts full responsibility for any further errors.
\section{Acknowledgements}
The author appreciates the encouragement and feedback from their thesis adviser Po-Lam Yung. They also appreciate Zane Li's friendly help and encouragement. The author also benefitted from discussions about polynomials with Snehinh Sen. Part of this project was completed while the author was visiting the National University of Singapore and the author wishes to thank Seng Kee Chua for his hospitality. The author was partially supported by Po-Lam Yung and Neil Trudinger's Discovery Project DP250103744 from the Australian Research Council. 

\end{document}